\documentclass{article} 

\usepackage{amsmath,amsthm,amssymb,amsfonts}

\usepackage{geometry} 
\newtheorem{Theorem}{Theorem}[section]
\theoremstyle{plain}
\newtheorem{Example}[Theorem]{Example}
\newtheorem{Lemma}[Theorem]{Lemma}

\newtheorem{Remark}[Theorem]{Remark}
\newtheorem{Corollary}[Theorem]{Corollary}
\newtheorem{Proposition}[Theorem]{Proposition}

\numberwithin{equation}{section}

\title{Limiting distribution of particles near the frontier in the catalytic branching Brownian motion}
\author{Sergey Bocharov\footnote{S.Bocharov:  
Department of Mathematics, Zhejiang University, Zheda Road, Hangzhou 310027, China, 
e-mail: bocharov@zju.edu.cn. The author is supported by NSFC grant (No.11731012)}}
\begin{document}
\maketitle
\begin{abstract}
We consider the model of branching Brownian motion with a single catalytic point at the origin and binary branching. We establish some fine results for the asymptotic behaviour of the numbers of particles  travelling at different speeds and give an explicit characterisation of the spatial distribution of particles travelling at the critical speed.
\end{abstract}
\section{Introduction and main results}
\subsection{Description of the model}
Branching Brownian motion with a single-point catalyst at the origin is a spatial population model in which individuals (referred to as particles) move in space according to the law of standard Brownian motion and reproduce themselves at a spatially-inhomogeneous branching rate $\beta \delta_0(\cdot)$, where $\delta_0(\cdot)$ is the Dirac delta measure and $\beta > 0$ is some constant.

More precisely, in such a process we start with a single particle at some initial location $x_0 \in \mathbb{R}$ at time $0$ whose position $X_t$ at time $t \geq 0$ up until the time it dies evolves like a standard Brownian motion. At a random time $T$ satisfying
\[
P^{x_0} \Big( T>t \big\vert \big( X_s \big)_{s \geq 0} \Big) = \mathrm{e}^{- \beta L_t},
\]
where $(L_t)_{t \geq 0}$ is the local time at $0$ of $(X_t)_{t \geq 0}$, the initial particle dies and is replaced with two new particles, which independently of each other and of the previous history stochastically continue the behaviour of their parent starting from time $T$ and position $X_T = 0$. That is, they move like Brownian motions, die after random times giving birth to two new particles each, etc.

Note that informally we may write $L_t = \int_0^t \delta_0(X_s) \mathrm{d}s$ thus justifying calling the branching rate $\beta \delta_0(\cdot)$. This is made precise by the theory of additive functionals of Brownian motion. See, for example, papers of Chen and Shiozawa \cite{CS07} and Shiozawa \cite{S08}, \cite{S18}, \cite{S19} where they study a large class of processes with branching rates which are allowed to be measures.

Let us mention that in the past catalytic branching processes have also been studied in the context of superprocesses (see for example papers of Dawson and Fleischmann \cite{DF94} and Engl\"ander and Turaev \cite{ET02}) and also in the context of branching random walks on integer lattices, both in discrete time (see, for example, a paper of Carmona and Hu \cite{CH14}) and continuous time (see, for example, a paper of Bulinskaya \cite{B18}).

Also, a closely related type of processes is branching Brownian motions with the branching rate given by either a compactly-supported function or a function decaying sufficiently fast at infinity (see, for example, papers of Koralov and Molchanov \cite{KM13}, Erickson \cite{E84} and Lalley and Sellke \cite{LS88}). 
\subsection{Notation and some earlier results}
Following a common practice we label the initial particle in the branching process by $\varnothing$ and all its ancestors according to the Ullam-Harris convention. In this way, for example, particle ``$\varnothing 1 2$" corresponds to child $2$ of child $1$ of the initial particle $\varnothing$.

We denote the set of all particles alive at time $t$ by $N_t$ and for every particle $u \in N_t$ we let $X^u_t$ be its spatial position at this time $t$. Furthermore, for any Borel set $A \subseteq \mathbb{R}$ we define
\[
N_t^A := \big\{ u \in N_t \ : \ X^u_t \in A \big\},
\]
the set of all particles located in the set $A$ at time $t$. 

We may, for example, take $A = [\lambda t, \infty)$ for some $\lambda > 0$, so that $N_t^{[\lambda t, \infty)}$ is the set of particles at time $t$ in the upper-half plane which are of distance at least $\lambda t$ from the origin, which we may also interpret as particles travelling at average speeds $\geq \lambda$.
It was shown in \cite{BH14} that if we define
\begin{equation}
\label{delta}
\Delta_\lambda := \left\{
\begin{array}{rl}
\frac{1}{2} \beta^2 - \beta \lambda& \text{if } \lambda \leq \beta, \\
- \frac{1}{2} \lambda^2& \text{if } \lambda \geq \beta, 
\end{array} \right.
\end{equation}
so that
\begin{equation*}
\Delta_\lambda \left\{
\begin{array}{rl}
> 0& \text{if } \lambda < \frac{\beta}{2}, \\
= 0& \text{if } \lambda = \frac{\beta}{2}, \\
< 0& \text{if } \lambda > \frac{\beta}{2}, 
\end{array} \right.
\end{equation*}
then the following results hold.

\noindent If $\lambda < \frac{\beta}{2}$ then
\begin{equation}
\label{old_subcrit}
\frac{1}{t} \log |N_t^{[\lambda t, \infty)}| \to \Delta_\lambda \ (>0) \qquad P^{x_0} \text{-a.s.}
\end{equation}
If $\lambda > \frac{\beta}{2}$ then
\begin{equation}
\label{old_supercrit}
|N_t^{[\lambda t, \infty)}| \to 0 \qquad P^{x_0} \text{-a.s.}
\end{equation}
and furthermore
\begin{equation}
\label{old_supercrit2}
\frac{1}{t} \log P^{x_0} \Big( |N_t^{[\lambda t, \infty)}| > 0 \Big) \to \Delta_\lambda \ (<0).
\end{equation}
In other words, the number of particles travelling at speeds $\lambda < \frac{\beta}{2}$ is growing exponentially while the number of particles travelling at speeds $\lambda > \frac{\beta}{2}$ is eventually $0$. It is then easily seen that if we define
\begin{equation}
\label{rightmost}
\mathcal{R}_t := \sup_{u \in N_t} X^u_t
\end{equation}
to be the position of the rightmost particle at time $t$ then 
\begin{equation}
\label{rightmost_limit}
\frac{\mathcal{R}_t}{t} \to \frac{\beta}{2} \qquad P^{x_0} \text{-a.s.}
\end{equation}
It was further shown in \cite{BH16} that for all $x \in \mathbb{R}$,
\begin{equation}
\label{rightmost_limit2}
P^{x_0} \Big(\mathcal{R}_t - \frac{\beta}{2}t \leq x\Big) \to E^{x_0} \exp \big\{ - \mathrm{e}^{- \beta x} M_\infty \big\},
\end{equation}
where $M_\infty$ is the strictly positive almost sure limit of the (square-integrable) martingale
\begin{equation}
\label{martingale}
M_t = \mathrm{e}^{- \frac{\beta^2}{2}t} \sum_{u \in N_t} \mathrm{e}^{- \beta |X^u_t|} \qquad \text{, } t \geq 0.
\end{equation}
Also, it was proved in a much more general setting in \cite{CS07} that for a suitable class of test functions $f(\cdot)$ it is true that
\begin{equation}
\label{slln}
\mathrm{e}^{- \frac{\beta^2}{2}t} \sum_{u \in N_t} f(X^u_t) \to M_\infty \int_{\mathbb{R}} f(x) \pi(\mathrm{d}x)
\qquad P \text{-a.s.},
\end{equation}
where
\begin{equation}
\label{invariant}
\pi(\mathrm{d}x) = \beta \mathrm{e}^{- \beta |x|} \mathrm{d}x
\end{equation}
(we don't normalise $\pi(\cdot)$ to be a probability measure). So, for example, taking $f(\cdot) = \mathbf{1}_A(\cdot)$ for a sufficiently nice set $A \subseteq \mathbb{R}$ one gets
\begin{equation}
\label{slln_sets}
\mathrm{e}^{- \frac{\beta^2}{2}t} |N^A_t| \to \pi(A) M_\infty \qquad P \text{-a.s.}
\end{equation}
Let us mention that versions of \eqref{old_subcrit} - \eqref{old_supercrit2} for a large class of branching Brownian motions were recently proved in \cite{S18} and \cite{S19}. Also, a while ago, versions of  \eqref{rightmost_limit} and \eqref{rightmost_limit2} for branching Brownian motions with branching rates given by continuous functions decaying sufficiently fast at infinity were proved in \cite{E84} and \cite{LS88} respectively.  Versions of \eqref{rightmost_limit} and \eqref{rightmost_limit2} for discrete -time  branching random walks on $\mathbb{Z}$ are available in \cite{CH14}.
\subsection{Main results}
Theorem \ref{main} below is the main result of this article. It essentially says that the distributions of particles around the critical lines $\frac{\beta}{2}t$ and $- \frac{\beta}{2}t$ converge to mixtures of Poisson point processes.
\begin{Theorem}
\label{main}
Take any $x_0 \in \mathbb{R}$, integers $n, m \geq 0$, integers $k_1, \cdots, k_n$, $l_1, \cdots, l_m \geq 0$ and Borel sets $A_1, \cdots, A_n$, $B_1, \cdots B_m$ such that $A_1, \cdots, A_n$ are mutually-disjoint, $B_1, \cdots, B_m$ are mutually-disjoint  and $\inf A_1, \cdots, \inf A_n$, $\inf B_1, \cdots, \inf B_m > - \infty$. 

For every Borel set $D \subseteq \mathbb{R}$ define
\[
\mu(D) := \int_D \beta \mathrm{e}^{- \beta y} \mathrm{d}y
\]
and, for convenience, let $k= k_1 + \cdots + k_n + l_1 + \cdots + l_m$, $A = \cup_{i = 1}^n A_i$ and 
$B = \cup_{j=1}^m B_j$. Then
\begin{align}
\label{eq_main}
\lim_{t \to \infty} &P^{x_0} \Big( 
\bigcap_{i=1}^n \big\{ \big\vert N_t^{A_i + \frac{\beta}{2} t} \big\vert = k_i \big\} \text{ , }
\bigcap_{j=1}^{m} \big\{ \big\vert N_t^{- B_j - \frac{\beta}{2} t} \big\vert = l_j \big\} \Big) \nonumber \\
= &E^{x_0} \Big[ \prod_{i = 1}^{n} \Big( \frac{\big( \mu(A_i) M_\infty \big)^{k_i}}{k_i !} 
\mathrm{e}^{- \mu(A_i) M_\infty} \Big) \prod_{j = 1}^{m} \Big( \frac{\big( \mu(B_j) M_\infty \big)^{l_j}}{l_j !} 
\mathrm{e}^{- \mu(B_j) M_\infty} \Big) \Big] \nonumber \\
= &\prod_{i = 1}^n \frac{\mu(A_i)^{k_i}}{k_i!} \prod_{j = 1}^m \frac{\mu(B_j)^{l_j}}{l_j!} E^{x_0} \Big[ M_\infty^k \mathrm{e}^{-(\mu(A) + \mu(B))M_\infty} \Big],
\end{align}
where in the above statement and everywhere else in this article for a Borel set $D \subseteq \mathbb{R}$ and a point $c \in \mathbb{R}$, $D + c = \big\{ x + c \ : \ x \in D \big\}$ and $ -D = \{ - x \ : \ x \in D \}$. We also adapt the conventions that $\inf \emptyset = \infty$, $\bigcap_{\emptyset}\{ \cdot \} = \Omega$ and $\prod_{\emptyset} (\cdot) = 1$.
\end{Theorem}
\begin{Remark}
\label{rem_main}
Let us note that we shall actually prove something slightly stronger than \eqref{eq_main}. Namely, that for $s(t)$ such that $s(t) \to \infty$ but $s(t) = o(t)$ as $t \to \infty$ it is true that
\begin{align}
\label{eq_main_cond}
&P^{x_0} \Big( \bigcap_{i=1}^n \big\{ \big\vert N_t^{A_i + \frac{\beta}{2} t} \big\vert = k_i \big\} \text{ , } \bigcap_{j=1}^{m} \big\{ \big\vert N_t^{- B_j - \frac{\beta}{2} t} \big\vert = l_j \big\} \ \Big\vert \ \mathcal{F}_{s(t)}\Big) \nonumber \\
\to &\ \prod_{i = 1}^{n} \Big( \frac{\big( \mu(A_i) M_\infty \big)^{k_i}}{k_i !} 
\mathrm{e}^{- \mu(A_i) M_\infty} \Big) \prod_{j = 1}^{m} \Big( \frac{\big( \mu(B_j) M_\infty \big)^{l_j}}{l_j !} 
\mathrm{e}^{- \mu(B_j) M_\infty} \Big) \qquad P^{x_0} \text{-a.s.} \nonumber\\
= &\prod_{i = 1}^n \frac{\mu(A_i)^{k_i}}{k_i!} \prod_{j = 1}^m \frac{\mu(B_j)^{l_j}}{l_j!} M_\infty^k \mathrm{e}^{-(\mu(A) + \mu(B))M_\infty},
\end{align}
where $(\mathcal{F}_t)_{t \geq 0}$ is the natural filtration of the branching process. Equation \eqref{eq_main} 
will then follow by bounded convergence.
\end{Remark}
Results of the type of Theorem \ref{main} are quite natural and have appeared in literature before. For example, 
the distribution of particles near the frontier in a branching Brownian motion with a spatially-homogeneous branching rate has been discussed a lot in recent years. See for example papers of A\"id\'ekon , Berestycki, Brunet and Shi \cite{ABBS13}, Arguin and Bovier \cite{AB15} and Brunet and Derrida \cite{BD11} to mention just a few (but note that the limiting distribution in such a model is a mixed decorated Poisson point process). The convergence of the distribution of particles near the frontier to a mixed Poisson point process in a  branching Brownian motion with a continuous branching rate decaying sufficiently fast at $\infty$ was also mentioned by Lalley and Sellke in \cite{LS88} although the argument they presented is quite different from ours.

Below we illustrate how Theorem \ref{main} can be applied.
\begin{Example}
\label{ex_1}
By analogy with the rightmost particle, for every $t \geq 0$, let us define
\[
\mathcal{L}_t := \inf_{u \in N_t} X^u_t,
\]
the position of the leftmost particle at time $t$. Then from \eqref{eq_main} we may recover the limiting joint distribution of $\mathcal{R}_t  - \frac{\beta}{2}t$ and $\mathcal{L}_t + \frac{\beta}{2}t$. Namely, for any $x^-$, $x^+ \in \mathbb{R}$ we have
\begin{align*}
&P^{x_0} \Big( \mathcal{L}_t + \frac{\beta}{2}t > x^- , \ \mathcal{R}_t - \frac{\beta}{2}t \leq x^+ \Big)\\
= &P^{x_0} \Big( \big\vert N_t^{(- \infty , - \frac{\beta}{2} t + x^-]} \big\vert = 0 , \
\big\vert N_t^{(\frac{\beta}{2}t + x^+, \infty)} \big\vert = 0 \Big)\\
\to &E^{x_0} \Big[ \exp \big\{ - \big(\mathrm{e}^{\beta x^-} + \mathrm{e}^{- \beta x^+} \big)M_\infty \big\} \Big]
\quad \text{ as } t \to \infty
\end{align*}
and hence
\begin{align*}
&P^{x_0} \Big( \mathcal{L}_t + \frac{\beta}{2}t \leq x^- , \ \mathcal{R}_t - \frac{\beta}{2}t \leq x^+ \Big)\\
= &P^{x_0} \Big(\mathcal{R}_t - \frac{\beta}{2}t \leq x^+ \Big) - P^{x_0} \Big( \mathcal{L}_t + \frac{\beta}{2}t > x^- , \ \mathcal{R}_t - \frac{\beta}{2}t \leq x^+ \Big)\\
\to &E^{x_0} \Big[ \exp \big\{ - \mathrm{e}^{- \beta x^+} M_\infty \big\} 
\Big( 1 - \exp \big\{ - \mathrm{e}^{ \beta x^-} M_\infty \big\} \Big) \Big] \quad \text{ as } t \to \infty.
\end{align*}
\end{Example}

\begin{Example}
\label{ex_2}
For every $t \geq 0$ and $n \in \mathbb{N}$ let $\mathcal{R}^{(n)}_t$ be the value of the $n$th largest spatial position of all the particles in the system at time $t$ so that $\mathcal{R}^{(1)} \equiv \mathcal{R}$. Then from  \eqref{eq_main} we derive the limiting distribution of $\mathcal{R}^{(n)}_t - \frac{\beta}{2} t$ generalising the earlier result \eqref{rightmost_limit2}. Namely, for any $x \in \mathbb{R}$ we have
\begin{align*}
&P^{x_0} \Big( \mathcal{R}^{(n)}_t - \frac{\beta}{2} t \leq x \Big)\\ 
= &P^{x_0} \Big( \big\vert N^{(\frac{\beta}{2}t + x, \infty)}\big\vert \leq n - 1 \Big)\\
\to &E^{x_0} \Big[ \Big( \sum_{k=0}^{n-1} \frac{(\mathrm{e}^{- \beta x} M_\infty)^k}{k!}\Big)
\exp \big\{ - \mathrm{e}^{- \beta x} M_\infty \big\} \Big]
\end{align*}
as $t \to \infty$.
\end{Example}
While proving our main result we shall also establish the following results regarding the asymptotic behaviour of the number of particles travelling at super- and subcritical speeds giving some finer versions of \eqref{old_subcrit} and \eqref{old_supercrit2}.
\begin{Proposition}[Subcritical speeds, $\lambda \in (0, \frac{\beta}{2})$]
\label{main3}
Take any $x_0 \in \mathbb{R}$, $\lambda \in (0, \frac{\beta}{2})$ and a Borel set $A \subset \mathbb{R}$ such that $\inf A > - \infty$. Then
\begin{equation}
\label{eq_main3}
\mathrm{e}^{- \Delta_\lambda t} |N_t^{A+\lambda t}| \to \mu(A) M_\infty \qquad \text{ in } P^{x_0}\text{-probability}.
\end{equation}
\end{Proposition}
\begin{Remark}
\label{rem_main3}
From the proof of Proposition \ref{main3} it will be apparent that convergence in \eqref{eq_main3} also holds almost surely along any sequence $(t_n)_{n \geq 1}$ such that $\frac{t_n}{(\log t_n)^\alpha} \to \infty$ for some appropriate choice of $\alpha > 0$.
\end{Remark}
\begin{Proposition}[Supercritical speeds, $\lambda \in (\frac{\beta}{2}, \beta)$]
\label{main2}
Take any $x_0 \in \mathbb{R}$, $\lambda \in (\frac{\beta}{2}, \beta)$ and Borel sets $A$, $B \subset \mathbb{R}$ such that $\inf A$, $\inf B > - \infty$. Then
\begin{equation}
\label{eq_main2}
\mathrm{e}^{-\Delta_\lambda t} P^{x_0} \Big( \big\vert N_t^{(A + \lambda t) \cup (-B - \lambda t)} \big\vert > 0 \Big) 
\to \big( \mu(A) + \mu(B)\big) \mathrm{e}^{- \beta |x_0|}.
\end{equation}
as $t \to \infty$.
\end{Proposition}
\begin{Remark}
\label{rem_main2}
The cases $\lambda = \beta$ and $\lambda > \beta$ will require separate analysis. Partial results are available in \cite{S19} (Theorem 3.7).
\end{Remark}
\subsection{Outline of the paper}
The article is organised as follows.

Subsection 2.1 is devoted to various first-moment calculations. In particular, we show there that given a Borel set $A$ such that $\inf A > - \infty$, it is is true that for large $t$, $s = o(t)$ and $x_0$, which is allowed to depend on $t$ to some extent,
\begin{equation} 
\label{approx1}
E^{x_0} \big\vert N_{t-s}^{A + \frac{\beta}{2}t} \big\vert \approx \mu(A) \mathrm{e}^{- \beta |x_0| - \frac{\beta^2}{2}s}.
\end{equation}
This is made precise in Corollary \ref{estimate1}.

In Subsection 2.2 we discuss second momemt calculations and in particular we show that
\begin{equation} 
\label{approx2}
E^{x_0} \big\vert N_{t-s}^{A + \frac{\beta}{2}t} \big\vert^2 = E^{x_0} \big\vert N_{t-s}^{A + \frac{\beta}{2}t} \big\vert + correction \ term,
\end{equation}
where we have a good control of the correction term.

In Subsection 2.3 we deduce from \eqref{approx1} and \eqref{approx2} that if $s \to \infty$ then
\begin{align}
\label{approx3}
&P^{x_0} \Big( \big\vert N_{t-s}^{A + \frac{\beta}{2}t} \big\vert = 0 \Big) \approx 1 - \mu(A) \mathrm{e}^{- \beta |x_0| - \frac{\beta^2}{2}s}, \nonumber\\
&P^{x_0} \Big( \big\vert N_{t-s}^{A + \frac{\beta}{2}t} \big\vert = 1 \Big) \approx \mu(A) \mathrm{e}^{- \beta |x_0| - \frac{\beta^2}{2}s}, \nonumber\\
&P^{x_0} \Big( \big\vert N_{t-s}^{A + \frac{\beta}{2}t} \big\vert > 1 \Big) \text{ becomes negligibly small.}
\end{align}
We also prove Proposition \ref{main2} there.

In Subsection 3.1 we prove \eqref{eq_main_cond} and consequently Theorem \ref{main} via the following argument. Take for simplicity a single set $A \subseteq \mathbb{R}$ and a non-negative integer $k$. Then note that from the Markov property
\[
\big\vert N_t^{A + \frac{\beta}{2}t} \big\vert = \sum_{u \in N_s} \big\vert N_{t-s}^{A + \frac{\beta}{2}t}(u) \big\vert,
\]
where, conditional on $\mathcal{F}_s$, $\big\vert N_{t-s}^{A + \frac{\beta}{2}t}(u) \big\vert$, $u \in N_s$ are independent copies of $\big\vert N_{t-s}^{A + \frac{\beta}{2}t} \big\vert$ in branching processes initiated from $X^u_s$, $u \in N_s$. Then from \eqref{approx3} we know that these are essentially Bernoulli random variables with conditional probabilities of success $\approx \mu(A) \mathrm{e}^{- \beta |X^u_s| - \frac{\beta^2}{2}s}$.

Then, making use of this observation and some other approximations, we get that 
\begin{align*}
&P^{x_0} \Big( \big\vert N_t^{A + \frac{\beta}{2}t} \big\vert = k \ \Big\vert \ \mathcal{F}_s \Big)\\
\approx &\frac{1}{k!} \sum_{(u_1, \cdots, u_k) \subseteq N_s} P^{X^{u_1}_s} \Big( \big\vert N_{t-s}^{A + \frac{\beta}{2}t} \big\vert = 1 \Big) \cdots P^{X^{u_k}_s} \Big( \big\vert N_{t-s}^{A + \frac{\beta}{2}t} \big\vert = 1 \Big)\\ 
&\qquad\qquad\qquad \times \prod_{u \neq u_1 , \cdots , u_k} P^{X^u_s} \Big( \big\vert N_{t-s}^{A + \frac{\beta}{2}t} \big\vert = 0 \Big)\\
\approx &\frac{1}{k!} \Big[\mu(A) \sum_{u_1 \in N_s} \mathrm{e}^{- \beta |X^{u_1}_s| - \frac{\beta^2}{2}s}\Big] \cdots \Big[\mu(A) \sum_{u_k \in N_s} \mathrm{e}^{- \beta |X^{u_k}_s| - \frac{\beta^2}{2}s}\Big]\\
&\qquad\qquad\qquad \times \prod_{u \in N_s} \Big( 1 - \mu(A) \mathrm{e}^{- \beta |X^u_s| - \frac{\beta^2}{2}s} \Big)\\
\approx &\frac{1}{k!} \mu(A)^k M_s^k \mathrm{e}^{- \mu(A) M_s}
\approx \frac{1}{k!} \mu(A)^k M_\infty^k \mathrm{e}^{- \mu(A) M_\infty},
\end{align*}
where the summation $\sum_{(u_1, \cdots, u_k) \subseteq N_s}$ is taken over all $k$-permutations of the set $N_s$.

The above argument makes it particularly clear that the Poisson distribution of particles near the frontier emerges from the generalised Poisson approximation to the Binomial.

We finish the paper with the proof of Theorem \ref{main3}, which we give in Subsection 3.2.
\section{Preliminary calculations}
In this section we derive various estimates for $\big\vert N_{t-s}^{A \pm \lambda t} \big\vert$ necessary for proofs of the main results.
\subsection{First moment calculations}

It is a common practice to extend the original probability space of the branching system by adding the spine process to it. The spine is an infinite line of descent which begins with the initial particle and whenever the particle presently in the spine dies one of its two children is chosen with probability $\frac{1}{2}$ to continue the spine independently of all the previous history. 

If we then let $\tilde{P}$ denote the extension of the original probability measure $P$ to this bigger probability space and if at every $t \geq 0$ we let $\xi_t$ denote the spatial position of the spine particle at time $t$ then one can see that the process $(\xi_t)_{t \geq 0}$ is a Brownian motion under $\tilde{P}$. Furthermore the following result is known to hold.
\begin{Lemma}[Many-to-One Lemma]
\label{many_to_one}
Let $f : \mathbb{R} \to \mathbb{R}$ be a sufficiently nice function (non-negative Borel measurable will be enough for us). Then
\begin{equation}
\label{eq_many_to_one}
E^{x_0} \Big[ \sum_{u \in N_t} f(X^u_t) \Big] = \tilde{E}^{x_0} \Big[ f(\xi_t) \mathrm{e}^{\beta \tilde{L}_t} \Big] \text{,}
\end{equation}
where $\tilde{E}$ is the expectation function corresponding to the probability measure $\tilde{P}$ and $(\tilde{L}_t)_{t \geq 0}$ is the local time at the origin of $(\xi_t)_{t \geq 0}$.
\end{Lemma}
For a detailed discussion of the spine approach to Many-to-One Lemma one may look at \cite{HH09} or \cite{HR}. For the derivation of \eqref{eq_many_to_one} without the spine construction see \cite{S08} (Lemma 3.3).

Let us also recall the $\tilde{P}$-martingale
\begin{equation}
\label{martingale_tilde}
\tilde{M}_t^\beta = \mathrm{e}^{- \beta |\xi_t| + \beta \tilde{L}_t - \frac{1}{2} \beta^2 t} 
= \mathrm{e}^{- \beta \int_0^t sgn (\xi_s) \mathrm{d} \xi_s - \frac{1}{2} \beta^2 t} \qquad \text{ , } t \geq 0
\end{equation}
discussed previously in \cite{BH14}. It is basically a Girsanov type martingale which, when used as the Radon-Nikodym derivative, has the effect of putting instantaneous drift $- \beta sgn(\cdot)$ (in other words, a drift of constant magnitude $\beta$ towards the origin) on $(\xi_t)_{t \geq 0}$ and from which the additive martingale \eqref{martingale} was constructed. The following result is taken from \cite{BS02} and we shall use it to simplify the evaluation of the right hand side in the formula \eqref{many_to_one}.
\begin{Proposition}
\label{transition_density}
Let $\tilde{Q}_\beta$ be the probability measure defined as 
\begin{equation}
\label{girsanov}
\frac{\mathrm{d}\tilde{Q}_\beta^{x_0}}{\mathrm{d} \tilde{P}^{x_0}} \Big\vert_{\sigma((\xi_s)_{0 \leq s \leq t})} = 
\dfrac{\tilde{M}_t^{\beta}}{\tilde{M}_0^{\beta}} = \mathrm{e}^{\beta |x_0|} \tilde{M}_t^\beta \qquad \text{ , } t \geq 0 \text{, } x_0 \in \mathbb{R}.
\end{equation}
Then under $\tilde{Q}_\beta$, $(\xi_t)_{t \geq 0}$ has the transition density (with respect to Lebesgue measure)
\begin{equation}
\label{density1}
p_t(x_0, x) = \frac{1}{\sqrt{2 \pi t}} \exp \Big\{ \beta \big( |x_0| - |x| \big) - \frac{\beta^2}{2} t - \frac{(x_0 - x)^2}{2t}
\Big\} + \beta \mathrm{e}^{- 2 \beta |x|} \Phi \Big( \frac{\beta t - |x_0| - |x|}{\sqrt{t}} \Big)
\end{equation}
so that for any set $A \subseteq \mathbb{R}$ and $t \geq 0$
\begin{equation}
\label{density2}
\tilde{Q}_\beta^{x_0} \big( \xi_t \in A \big) = \int_A p_t(x_0, x) \mathrm{d}x .
\end{equation}
\end{Proposition}

From Lemma \ref{many_to_one} and Proposition \ref{transition_density} we derive the following exact expression for the expected number of particles in the set $A$ at time $t$.
\begin{Proposition}
\label{first_moment}
For any $x_0 \in \mathbb{R}$, a Borel set $A \subseteq \mathbb{R}$ and $t \geq 0$ we have
\begin{equation}
\label{first_moment_eq}
E^{x_0} \big\vert N_t^A \big\vert = \int_A \frac{1}{\sqrt{2 \pi t}} \mathrm{e}^{- \frac{(x_0 - x)^2}{2t}} \mathrm{d}x 
+ \beta \mathrm{e}^{- \beta |x_0| + \frac{\beta^2}{2}t} \int_A \mathrm{e}^{- \beta |x|} \Phi \Big( 
\frac{\beta t - |x_0| - |x|}{\sqrt{t}} \Big) \mathrm{d}x .
\end{equation}
\end{Proposition}
\begin{proof}
Applying Lemma \ref{many_to_one} and the change of measure \eqref{girsanov} we obtain
\begin{align*}
E^{x_0} \big\vert N_t^A \big\vert = &E^{x_0} \Big[ \sum_{u \in N_t} \mathbf{1}_A(X^u_t) \Big]\\
= &\tilde{E}^{x_0} \Big[ \mathbf{1}_A(\xi_t) \mathrm{e}^{\beta \tilde{L}_t} \Big]\\
= &\tilde{Q}_\beta^{x_0} \Big[ \frac{\tilde{M}^\beta_0}{\tilde{M}^\beta_t} \mathbf{1}_A(\xi_t) \mathrm{e}^{\beta \tilde{L}_t} \Big]\\
= &\mathrm{e}^{- \beta|x_0| + \frac{\beta^2}{2}t} \tilde{Q}_\beta^{x_0} \Big[ \mathbf{1}_A(\xi_t) \mathrm{e}^{\beta |\xi_t|} \Big].
\end{align*}
Then substituting the formula for $\tilde{Q}_\beta$-transition density of $(\xi_t)_{t \geq 0}$ \eqref{density1} we get the sought expression:
\begin{align*}
\mathrm{e}^{- \beta|x_0| + \frac{\beta^2}{2}t} \tilde{Q}_\beta^{x_0} \Big[ \mathbf{1}_A(\xi_t) \mathrm{e}^{\beta |\xi_t|} \Big] = &\mathrm{e}^{- \beta|x_0| + \frac{\beta^2}{2}t} \int_A \mathrm{e}^{\beta |x|} p_t(x_0, x) \mathrm{d}x\\
= &\int_A \frac{1}{\sqrt{2 \pi t}} \mathrm{e}^{- \frac{(x_0 - x)^2}{2t}} \mathrm{d}x\\ 
&+ \beta \mathrm{e}^{- \beta |x_0| + \frac{\beta^2}{2}t} \int_A \mathrm{e}^{- \beta |x|} \Phi \Big( 
\frac{\beta t - |x_0| - |x|}{\sqrt{t}} \Big) \mathrm{d}x .
\end{align*}
\end{proof}
Let us now derive a number of estimates from \eqref{first_moment_eq} for later use.
\begin{Corollary}
\label{exp_pop}
For any $x_0 \in \mathbb{R}$ and $t \geq 0$
\begin{equation}
\label{eq_exp_pop}
E^{x_0}|N_t| \leq 1 + 2 \mathrm{e}^{- \beta |x_0| + \frac{\beta^2}{2}t}.
\end{equation}
\end{Corollary}
\begin{proof}
By substituting $A = \mathbb{R}$ in \eqref{first_moment_eq} and using symmetry in the second integral we get 
\begin{align*}
E^{x_0}|N_t| = &1 + 2 \beta \mathrm{e}^{- \beta |x_0| + \frac{\beta^2}{2}t} \int_0^\infty \mathrm{e}^{- \beta x} 
\Phi \Big( \frac{\beta t - |x_0| - x}{\sqrt{t}} \Big) \mathrm{d}x\\
\leq &1 + 2 \beta \mathrm{e}^{- \beta |x_0| + \frac{\beta^2}{2}t} \int_0^\infty \mathrm{e}^{- \beta x} \mathrm{d}x\\
= &1 + 2 \mathrm{e}^{- \beta |x_0| + \frac{\beta^2}{2}t}.
\end{align*}
Out of interest one may also evaluate the above integral exactly and find that
\[
E^{x_0}|N_t| = 1 + 2 \mathrm{e}^{- \beta |x_0| + \frac{\beta^2}{2}t} \Phi \Big( \beta \sqrt{t} - \frac{|x_0|}{\sqrt{t}}\Big) - 2 \Phi \Big( - \frac{|x_0|}{\sqrt{t}} \Big).
\]
\end{proof}

\begin{Corollary}
\label{exp_estimate}
Let $A$, $B \subseteq \mathbb{R}$ be Borel sets such that $\inf A$, $\inf B \geq 0$ and suppose that $x_0 = 0$. Then for any $t \geq 0$
\begin{equation}
\label{eq_exp_estimate}
E \big\vert N_t^{A\cup(-B)} \big\vert \leq \big( \mathrm{e}^{- \beta \inf A} + \mathrm{e}^{- \beta \inf B}\big) \mathrm{e}^{\frac{\beta^2}{2}t}.
\end{equation}
\end{Corollary}
\begin{proof}
From the fact that $A \subseteq [\inf A , \infty)$, equation \eqref{first_moment_eq} and the integration-by-parts formula we get that
\begin{align*}
E|N_t^A| \leq &E|N_t^{[\inf A, \infty)}|\\
= &\int_{\inf A}^\infty \frac{1}{\sqrt{2 \pi t}} \mathrm{e}^{- \frac{x^2}{2t}} \mathrm{d}x + \beta 
\mathrm{e}^{\frac{\beta^2}{2}t} \int_{\inf A}^\infty \mathrm{e}^{- \beta x} \Phi \Big( \beta \sqrt{t} - \frac{x}{\sqrt{t}}\Big) \mathrm{d}x\\
= &\Phi\Big( \beta \sqrt{t} - \frac{\inf A}{\sqrt{t}}\Big) \mathrm{e}^{\frac{\beta^2}{2}t - \beta \inf A}\\
\leq &\mathrm{e}^{\frac{\beta^2}{2}t - \beta \inf A}.
\end{align*}
By symmetry it follows that
\[
E|N_t^{-B}| = E|N_t^B| \leq \mathrm{e}^{\frac{\beta^2}{2}t - \beta \inf B}.
\]
Then since $A$ and $-B$ are disjoint we have that
\[
E \big\vert N_t^{A\cup(-B)} \big\vert = E|N_t^A| + E|N_t^{-B}| \leq 
\big( \mathrm{e}^{- \beta \inf A} + \mathrm{e}^{- \beta \inf B}\big) \mathrm{e}^{\frac{\beta^2}{2}t}.
\]
\end{proof}

\begin{Corollary}
\label{estimate_exp}
Take any real numbers $\lambda \in (0, \beta)$ and $K > 0$, sets $A$, $B \subseteq \mathbb{R}$ satisfying $\inf A$, $\inf B > - \infty$ and a function $s : [0, \infty) \to [0, \infty)$ such that $s(t) = o(t)$ as $t \to \infty$ and $t - s(t) \geq 0$ for all $t \geq 0$. 

Then for any choice of the above quantities there exist functions $\theta_1(\cdot)$, $\theta_2(\cdot) : [0, \infty) \to [0, \infty)$ satisfying $\theta_1(t)$, $\theta_2(t) \to 1$ as $t \to \infty$ such that for any $t \geq 0$ and $x_0 \in \mathbb{R}$ with $|x_0| < K s(t)$ it is true that
\begin{align}
\label{estimate1}
&E^{x_0} \big\vert N_{t-s}^{(A + \lambda t)\cup(-B - \lambda t)} \big\vert \geq \big( \mu(A) + \mu(B) \big) \mathrm{e}^{- \beta |x_0| - \frac{\beta^2}{2} s + \Delta_\lambda t} \theta_1(t), \nonumber\\
&E^{x_0} \big\vert N_{t-s}^{(A + \lambda t)\cup(-B - \lambda t)} \big\vert \leq \big( \mu(A) + \mu(B) \big) \mathrm{e}^{- \beta |x_0| - \frac{\beta^2}{2} s + \Delta_\lambda t} \theta_2(t),
\end{align}
where $s = s(t)$. 
\end{Corollary}
\begin{proof}
Let us first establish \eqref{estimate1} for $E^{x_0} \big\vert N_{t-s}^{A + \lambda t} \big\vert$. 
Take $\lambda$, $K$, $A$ and $s(\cdot)$ as above. From \eqref{first_moment_eq} we have that for all 
$t \geq 0$
\begin{align}
\label{I_II}
E^{x_0} \big\vert N_{t-s}^{A+\lambda t} \big\vert = &\int_{A + \lambda t} \frac{1}{\sqrt{2 \pi (t-s)}} \mathrm{e}^{- \frac{(x_0 - x)^2}{2(t-s)}} \mathrm{d}x \nonumber\\ 
+ &\mathrm{e}^{- \beta |x_0| + \frac{\beta^2}{2}(t-s)} \int_{A + \lambda t} \Phi \Big( 
\frac{\beta (t-s) - |x_0| - |x|}{\sqrt{t-s}} \Big) \beta \mathrm{e}^{- \beta |x|} \mathrm{d}x.
\end{align}
Let us denote the first integral on the RHS of \eqref{I_II} by (I) and the second one by (II). Then for $|x_0| < K s$ we have
\begin{align*}
(I) = &\int_{A + \lambda t} \frac{1}{\sqrt{2 \pi (t-s)}} \mathrm{e}^{- \frac{(x_0 - x)^2}{2(t-s)}} \mathrm{d}x\\ 
= &\mathbb{P} \Big( \mathcal{N}(x_0, t - s) \in A + \lambda t\Big)\\
\leq &\mathbb{P} \Big( \mathcal{N}(x_0, t - s) \geq \inf A + \lambda t\Big)\\
\leq &\mathbb{P} \Big( \mathcal{N}(0,1) \geq \frac{\inf A - K s + \lambda t}{\sqrt{t-s}} \Big),
\end{align*}
where $\mathcal{N}(\mu, \sigma^2)$ is a random variable with mean $\mu$ and variance $\sigma^2$. Thus using the estimate of the tail of the normal distribution as well as the defining property of $s(\cdot)$ we get
\begin{equation}
\label{integral1}
(I) \leq \theta_3(t) \mathrm{e}^{- \frac{\lambda^2}{2}t},
\end{equation}
where $\theta_3(\cdot)$ is some function with subexponential growth rate (that is, for any $\delta > 0$, 
$\theta_3(t) \mathrm{e}^{- \delta t} \to 0$ as $t \to \infty$).

Also, for all $t$ large enough so that $\inf A + \lambda t > 0$ we have that 
\begin{align*}
(II) = &\mathrm{e}^{- \beta |x_0| + \frac{\beta^2}{2}(t-s)} \int_{A + \lambda t} \Phi \Big( 
\frac{\beta (t-s) - |x_0| - x}{\sqrt{t-s}} \Big) \beta \mathrm{e}^{- \beta x} \mathrm{d}x\\
= &\mathrm{e}^{- \beta |x_0| - \frac{\beta^2}{2}s + \Delta_\lambda t} \int_A \Phi \Big( 
\frac{(\beta - \lambda)t - \beta s - |x_0| - y}{\sqrt{t-s}} \Big) \beta \mathrm{e}^{- \beta y} \mathrm{d}y
\end{align*}
by making substitution $x = y + \lambda t$ in the last line. We then observe that since $\Phi(\cdot) \leq 1$,
\[
\int_A \Phi \Big( \frac{(\beta - \lambda)t - \beta s - |x_0| - y}{\sqrt{t-s}} \Big) \beta \mathrm{e}^{- \beta y} \mathrm{d}y \leq \mu(A)
\]
and that for any $\epsilon \in (0 , \beta - \lambda)$ and $x_0$ such that $|x_0| < K s$,
\begin{align*}
&\int_A \Phi \Big( \frac{(\beta - \lambda)t - \beta s - |x_0| - y}{\sqrt{t-s}} \Big) \beta \mathrm{e}^{- \beta y} \mathrm{d}y\\ 
\geq &\int_{A \cap (- \infty, \epsilon t]} \Phi \Big( \frac{(\beta - \lambda)t - \beta s - |x_0| - y}{\sqrt{t-s}} \Big) \beta \mathrm{e}^{- \beta y} \mathrm{d}y\\
\geq &\Phi \Big( \frac{(\beta - \lambda - \epsilon)t - (\beta + K) s}{\sqrt{t-s}} \Big) \mu \big( A \cap (- \infty, \epsilon t] \big)\\
\to &\mu(A) \qquad \text{ as } t \to \infty.
\end{align*}
Hence for all t large enough
\[
\theta_4(t) \mu(A) \mathrm{e}^{- \beta |x_0| + \Delta_\lambda t - \frac{\beta^2}{2}s} \leq (II) 
\leq \mu(A) \mathrm{e}^{- \beta |x_0| + \Delta_\lambda t - \frac{\beta^2}{2}s}
\]
for some function $\theta_4(\cdot)$ such that $\theta_4(t) \to 1$ as $t \to \infty$.

Noting that $- \frac{\lambda^2}{2} < \Delta_\lambda$ we see from \eqref{integral1} that (I) makes a vanishigly small contribution to \eqref{I_II} thus establishing inequality \eqref{estimate1} for $E^{x_0} \big\vert N_{t-s}^{A+\lambda t} \big\vert= (I) + (II)$.

Then by symmetry we have 
\[
E^{x_0} \big\vert N_{t-s}^{-B - \lambda t} \big\vert = E^{-x_0} \big\vert N_{t-s}^{B + \lambda t} \big\vert. 
\]
So $E^{x_0} \big\vert N_{t-s}^{-B - \lambda t} \big\vert$ satisfies inequalities \eqref{estimate1} as well. Also, for $t$ large enough $A + \lambda t$ and $-B - \lambda t$ are disjoint so that 
\[
E^{x_0} \big\vert N_{t-s}^{(A + \lambda t)\cup(-B - \lambda t)} \big\vert = 
E^{x_0} \big\vert N_{t-s}^{A + \lambda t} \big\vert + E^{x_0} \big\vert N_{t-s}^{-B - \lambda t} \big\vert
\]
proving inequalities \eqref{estimate1} for $E^{x_0} \big\vert N_{t-s}^{(A + \lambda t)\cup(-B - \lambda t)} \big\vert$ .
\end{proof}
\subsection{Second moment calculations}
It is also possible to extend the original probability space of the branching process by adding two independent spine processes to it. If we let $\tilde{P}_2$ denote the extension of the original probability measure $P$ to this larger probability space and if for every $t \geq 0$ we let $\xi_t^{(1)}$ and $\xi_t^{(2)}$ denote the spatial positions of the two spine particles at time $t$ then one can check that $(\xi_t^{(1)})_{t \geq 0}$ and $(\xi_t^{(2)})_{t \geq 0}$ are two (correlated) Brownian motions under $\tilde{P}_2$. One can then write the formula for the second moment of $\sum_{u \in N_t} f(X^u_t)$ in terms of these two spine processes which, as shown in \cite{BH16}, reduces to the following result.
\begin{Lemma}[Many-to-Two Lemma]
\label{many_to_two}
Let $f : \mathbb{R} \to \mathbb{R}$ be a sufficiently nice function as in Lemma \ref{many_to_one}. Then
\begin{equation}
\label{eq_many_to_two}
E^{x_0} \Big[ \sum_{u \in N_t} f(X^u_t) \Big]^2 = S^{x_0}_{f^2}(t) + 
2 \tilde{E}^{x_0}_2 \Big( \int_0^t \big[ S^0_f(t - \tau) \big]^2 \mathrm{d} \big( \mathrm{e}^{\beta \tilde{L}_\tau^{(1)}} \big)   \Big) \text{,}
\end{equation}
where $\tilde{E}_2$ is the expectation function corresponding to the probability measure $\tilde{P}_2$, $(\tilde{L}_t^{(1)})_{t \geq 0}$ is the local time at the origin of $(\xi_t^{(1)})_{t \geq 0}$ and 
\begin{equation}
\label{s_f}
S^{x_0}_f(t) = E^{x_0} \Big[ \sum_{u \in N_t} f(X^u_t) \Big]
\end{equation}
(which can be computed using \eqref{eq_many_to_one})
\end{Lemma}
Alternative derivation of \eqref{eq_many_to_two} without the spine construction is available in \cite{S08} (Lemma 3.3). Note that in our model it doesn't matter whether to write $S^0_f(\cdot)$ or $S^{\xi_\tau^{(1)}}_f(\cdot)$ in the integrand since the integrator is only growing on the zero set of $\xi^{(1)}$.
\begin{Proposition}
\label{square_estimate}
For any $x_0 \in \mathbb{R}$, $\lambda \in (0, \beta)$, $A$, $B \subseteq \mathbb{R}$, $t \geq 0$ such that 
$\inf A + \lambda t$, $\inf B + \lambda t \geq 0$ and $s \in [0, t]$ we have that
\begin{equation}
\label{square_estimate_eq}
E^{x_0} \Big\vert N_{t-s}^{(A + \lambda t) \cup (-B - \lambda t)} \Big\vert^2 \leq E^{x_0} \Big\vert N_{t-s}^{(A + \lambda t) \cup (-B - \lambda t)} \Big\vert 
+ C \mathrm{e}^{- \beta |x_0| - \beta^2 s + 2 \Delta_\lambda t},
\end{equation}.
where $C$ is some positive constant (which depends on $A$ and $B$ only).
\end{Proposition}
\begin{proof}
Taking $t$ to be $t-s$ and $f(\cdot) = \mathbf{1}_{\{(A +\lambda t) \cup (-B - \lambda t)\}}(\cdot)$ in Lemma \ref{many_to_two} we get
\begin{align*}
E^{x_0} \Big\vert N_{t-s}^{(A + \lambda t)\cup(-B - \lambda t)} \Big\vert^2 = &E^{x_0} \Big\vert N_{t-s}^{(A + \lambda t)\cup(-B -\lambda t)} \Big\vert\\ 
+ &2 \tilde{E}^{x_0}_2 \Big( \int_0^t \Big[ E \big\vert N_{t-s-\tau}^{(A + \lambda t)\cup(-B - \lambda t)}\big\vert
\Big]^2 \mathrm{d} \big( \mathrm{e}^{\beta \tilde{L}_\tau^{(1)}} \big) \Big).
\end{align*}
Then from Corollary \ref{exp_estimate} we get 
\begin{align*}
&\tilde{E}^{x_0}_2 \Big( \int_0^t \Big[ E \big\vert N_{t-s-\tau}^{(A + \lambda t)\cup(-B - \lambda t)}\big\vert
\Big]^2 \mathrm{d} \big( \mathrm{e}^{\beta \tilde{L}_\tau^{(1)}} \big) \Big)\\ 
\leq &\tilde{E}^{x_0}_2 \Big( \int_0^t \Big[ \big( \mathrm{e}^{- \beta(\inf A + \lambda t)} + \mathrm{e}^{- \beta(\inf B + \lambda t)}\big) \mathrm{e}^{\frac{\beta^2}{2}(t-s-\tau)} \Big]^2 \mathrm{d} \big( \mathrm{e}^{\beta \tilde{L}_\tau^{(1)}} \big) \Big)\\
= &\big( \mathrm{e}^{-\beta \inf A} + \mathrm{e}^{-\beta \inf B}\big)^2
\mathrm{e}^{2 \Delta_\lambda t - \beta^2 s} \tilde{E}^{x_0} \Big( \int_0^t \mathrm{e}^{- \beta^2 \tau} \mathrm{d} \big( \mathrm{e}^{\beta \tilde{L}_\tau^{(1)}} \big)\Big).
\end{align*}
Then applying the integration-by-parts formula we get that 
\begin{align*}
\tilde{E}^{x_0} \Big( \int_0^t \mathrm{e}^{- \beta^2 \tau} \mathrm{d} \big( \mathrm{e}^{\beta \tilde{L}_\tau^{(1)}} \big)\Big) = 
&\tilde{E}^{x_0} \Big( \Big[ \mathrm{e}^{- \beta^2 \tau} \mathrm{e}^{\beta \tilde{L}_\tau^{(1)}} \Big]_0^t
+ \int_0^t \beta^2 \mathrm{e}^{- \beta^2 \tau} \mathrm{e}^{\beta \tilde{L}_\tau^{(1)}} \mathrm{d} \tau\Big)\\
= &\Big[ \mathrm{e}^{- \beta^2 \tau} \tilde{E}^{x_0} \big( \mathrm{e}^{\beta \tilde{L}_\tau^{(1)}} \big) \Big]_0^t
+ \int_0^t \beta^2 \mathrm{e}^{- \beta^2 \tau} \tilde{E}^{x_0} \big(\mathrm{e}^{\beta \tilde{L}_\tau^{(1)}} \big)\mathrm{d} \tau.
\end{align*}
From Corollary \ref{exp_pop} we also get that
\[
\tilde{E}^{x_0} \big( \mathrm{e}^{\beta \tilde{L}_\tau^{(1)}} \big) = E^{x_0} |N_\tau| \leq 
1 + 2 \mathrm{e}^{- \beta |x_0| + \frac{\beta^2}{2}\tau}.
\]
Hence
\begin{align*}
&\Big[ \mathrm{e}^{- \beta^2 \tau} \tilde{E}^{x_0} \big( \mathrm{e}^{\beta \tilde{L}_\tau^{(1)}} \big) \Big]_0^t
+ \int_0^t \beta^2 \mathrm{e}^{- \beta^2 \tau} \tilde{E}^{x_0} \big(\mathrm{e}^{\beta \tilde{L}_\tau^{(1)}} \big)\mathrm{d} \tau\\
\leq &\Big[ \mathrm{e}^{- \beta^2 \tau}\big( 1 + 2 \mathrm{e}^{- \beta |x_0| + \frac{\beta^2}{2}\tau} \big) \Big]_0^t
+ \int_0^t \beta^2 \mathrm{e}^{- \beta^2 \tau} \big(1 + 2 \mathrm{e}^{- \beta |x_0| + \frac{\beta^2}{2}\tau}\big)\mathrm{d} \tau\\
= &\int_0^t \mathrm{e}^{- \beta^2 \tau} \frac{\partial}{\partial \tau} \Big( 1 + 2 \mathrm{e}^{- \beta |x_0| + \frac{\beta^2}{2}\tau}\Big) \mathrm{d}\tau\\
= &2 \mathrm{e}^{- \beta |x_0|} \int_0^t \frac{\beta^2}{2} \mathrm{e}^{- \frac{\beta^2}{2} \tau} \mathrm{d} \tau
\leq 2 \mathrm{e}^{- \beta |x_0|},
\end{align*}
which establishes \eqref{square_estimate_eq} with $C = 2 \big( \mathrm{e}^{-\beta \inf A} + \mathrm{e}^{-\beta \inf B}\big)^2$. 
\end{proof}
\subsection{Probability estimates}
\begin{Proposition}
\label{prob_estimates}
Take any real numbers $\lambda \in [\frac{\beta}{2}, \beta)$ and $K > 0$, sets $A$, $B \subset \mathbb{R}$ satisfying $\inf A$, $\inf B > - \infty$ and a function $s : [0, \infty) \to [0, \infty)$ such that $s(t) = o(t)$ as $t \to \infty$, $t - s(t) \geq 0$ for all $t \geq 0$ and if $\lambda = \frac{\beta}{2}$ then $s(t) \to \infty$ as $t \to \infty$. 

Then for any choice of the above quantities there exist functions $\theta_{5}(\cdot)$, $\theta_{6}(\cdot)$, $\theta_{7}(\cdot)$, $\theta_8(\cdot) : [0, \infty) \to [0, \infty)$ satisfying $\theta_5(t)$, $\theta_6(t)$, $\theta_7(t)$, $\theta_8(t) \to 1$ as $t \to \infty$ such that for any $t \geq 0$ and $x_0 \in \mathbb{R}$ with $|x_0| < K s(t)$ it is true that
\begin{align}
&P^{x_0} \Big( \big\vert N_{t-s}^{(A + \lambda t) \cup (-B -\lambda t)}\big\vert = 0 \Big) \geq 
1 - \big( \mu(A) + \mu(B) \big) \mathrm{e}^{- \beta |x_0| - \frac{\beta^2}{2} s + \Delta_\lambda t} \theta_5(t),\label{estimate2}\\
&P^{x_0} \Big( \big\vert N_{t-s}^{(A + \lambda t) \cup (-B -\lambda t)}\big\vert = 0 \Big) \leq 
1 - \big( \mu(A) + \mu(B) \big) \mathrm{e}^{- \beta |x_0| - \frac{\beta^2}{2} s + \Delta_\lambda t} \theta_6(t),\label{estimate2a}\\
&P^{x_0} \Big( \big\vert N_{t-s}^{(A + \lambda t) \cup (-B -\lambda t)}\big\vert = 1 \Big) \geq \big( \mu(A) + \mu(B) \big) \mathrm{e}^{- \beta |x_0| - \frac{\beta^2}{2} s + \Delta_\lambda t} \theta_7(t),\label{estimate3}\\
&P^{x_0} \Big( \big\vert N_{t-s}^{(A + \lambda t) \cup (-B -\lambda t)}\big\vert = 1 \Big) \leq \big( \mu(A) + \mu(B) \big) \mathrm{e}^{- \beta |x_0| - \frac{\beta^2}{2} s + \Delta_\lambda t} \theta_8(t),\label{estimate3a}\\
&P^{x_0} \Big( \big\vert N_{t-s}^{(A + \lambda t) \cup (-B -\lambda t)}\big\vert > 1 \Big) \leq C \mathrm{e}^{- \beta |x_0| - \beta^2 s + 2 \Delta_\lambda t},
\label{estimate4}
\end{align}
where $s  = s(t)$ and $C$ is a positive constant (the same one as in Proposition \eqref{square_estimate}).
\end{Proposition}
Note that from Proposition \ref{prob_estimates} we immediately prove Theorem \ref{main2} by taking $s(\cdot) \equiv 0$ in \eqref{estimate2} and \eqref{estimate2a}.
\begin{proof}
Inequality \eqref{estimate4} follows from \eqref{square_estimate_eq} and the trivial fact that if $X$ is a random variable supported on $\{0, 1, 2, \cdots \}$ then $\mathbb{P}(X>1) \leq \mathbb{E}X^2 - \mathbb{E}X$.

From Markov's inequality and \eqref{estimate1} we have
\begin{align}
\label{estimate5}
P^{x_0} \Big( \big\vert N_{t-s}^{(A + \lambda t) \cup (-B -\lambda t)}\big\vert > 0 \Big) \leq &E^{x_0} \big\vert N_{t-s}^{(A + \lambda t) \cup (-B -\lambda t)} \big\vert \nonumber\\ 
\leq &\big( \mu(A) + \mu(B)\big) \mathrm{e}^{- \beta |x_0| - \frac{\beta^2}{2} s + \Delta_\lambda t} \theta_2(t).
\end{align}
From Paley-Zygmund's inequality, \eqref{estimate1} and \eqref{square_estimate_eq} we have (as long as $\mu(A)+ \mu(B)>0$) that
\begin{align}
\label{estimate6}
P^{x_0} \Big( \big\vert N_{t-s}^{(A + \lambda t) \cup (-B -\lambda t)}\big\vert > 0 \Big) \geq &\frac{\Big[ E^{x_0} \big\vert N_{t-s}^{(A + \lambda t) \cup (-B -\lambda t)} \big\vert\Big]^2}{E^{x_0} \big\vert N_{t-s}^{(A + \lambda t) \cup (-B -\lambda t)} \big\vert^2}\nonumber\\ 
\geq &\frac{\Big[ \big( \mu(A) + \mu(B)\big) \mathrm{e}^{- \beta |x_0| - \frac{\beta^2}{2} s + \Delta_\lambda t} \theta_1(t)\Big]^2}{\big( \mu(A) + \mu(B)\big) \mathrm{e}^{- \beta |x_0| - \frac{\beta^2}{2} s + \Delta_\lambda t} \theta_2(t) + C \mathrm{e}^{- \beta |x_0| - \beta^2 s + 2 \Delta_\lambda t}}\nonumber\\
= &\big( \mu(A) + \mu(B)\big) \mathrm{e}^{- \beta |x_0| - \frac{\beta^2}{2} s + \Delta_\lambda t} \Big[
\frac{\theta_1(t)^2}{\theta_2(t) + \frac{C}{\mu(A) + \mu(B)}\mathrm{e}^{-\frac{\beta^2}{2}s + \Delta_\lambda t}}\Big].
\end{align}
Substituting inequalities \eqref{estimate5} and \eqref{estimate6} in
\[
P^{x_0} \Big( \big\vert N_{t-s}^{(A + \lambda t) \cup (-B -\lambda t)}\big\vert = 0 \Big) =  1 - P^{x_0} \Big( \big\vert N_{t-s}^{(A + \lambda t) \cup (-B -\lambda t)}\big\vert > 0 \Big)
\]
we establish \eqref{estimate2} with $\theta_5(t) = \theta_2(t)$ and $\theta_6(t) = \theta_1(t)^2\big(\theta_2(t) + \frac{C}{\mu(A)+\mu(B)}\mathrm{e}^{-\frac{\beta^2}{2}s + \Delta_\lambda t}\big)^{-1}$ (and in the case $\mu(A) + \mu(B) =0$ we can just take $\theta_6 \equiv 1$).

Finally, substituting \eqref{estimate2} and \eqref{estimate4} in
\begin{align*}
P^{x_0} \Big( \big\vert N_{t-s}^{(A + \lambda t) \cup (-B -\lambda t)}\big\vert = 1 \Big) = &P^{x_0} \Big( \big\vert N_{t-s}^{(A + \lambda t) \cup (-B -\lambda t)}\big\vert > 0 \Big)- P^{x_0} \Big( \big\vert N_{t-s}^{(A + \lambda t) \cup (-B -\lambda t)}\big\vert > 1 \Big)\\
= &1 - P^{x_0} \Big( \big\vert N_{t-s}^{(A + \lambda t) \cup (-B -\lambda t)}\big\vert = 0 \Big) - P^{x_0} \Big( \big\vert N_{t-s}^{(A + \lambda t) \cup (-B -\lambda t)}\big\vert > 1 \Big)
\end{align*}
we establish \eqref{estimate3} with $\theta_8(t) = \theta_5(t)$ and $\theta_7(t) = \theta_6(t) - \frac{C}{\mu(A)+\mu(B)} \mathrm{e}^{- \frac{\beta^2}{2}s + \Delta_\lambda t}$ if $\mu(A) + \mu(B)>0$ (and in the case $\mu(A) + \mu(B) = 0$ we can just take $\theta_7 \equiv 1$).
\end{proof}

\section{Proof of the main results}
\subsection{Proof of Theorem \ref{main}}
Take any $x_0 \in \mathbb{R}$, integers $n, m \geq 0$, $k_1, \cdots, k_n$, $l_1, \cdots, l_m \geq 0$ and Borel sets $A_1, \cdots, A_n$, $B_1, \cdots B_m$ such that $A_1, \cdots, A_n$ are mutually-disjoint, $B_1, \cdots, B_m$ are mutually-disjoint  and $\inf A_1, \cdots, \inf A_n$, $\inf B_1, \cdots, \inf B_m > - \infty$. 

For our convenience let us define 
\begin{align*}
k:= &k_1 + \cdots +k_n + l_1 + \cdots + l_m,\\
A:= &\bigcup_{i=1}^n A_i, \ B:= \bigcup_{j=1}^m B_j.
\end{align*}
Let us fix a function $s(\cdot)$ such that $s(t) \to \infty$ but $s(t) = o(t)$ as $t \to \infty$ and $t - s(t) \geq 0$ for all $t \geq 0$ (e.g. $s(t) = \min \{ \sqrt{t}, t\}$). We shall write $s$ instead of $s(t)$ to lighten the notation.

Our aim is to prove  that 
\begin{align}
\label{eq_main_proof}
&P^{x_0} \Big( \bigcap_{i=1}^n \big\{ \big\vert N_t^{A_i + \frac{\beta}{2} t} \big\vert = k_i \big\} \text{ , }
\bigcap_{j=1}^{m} \big\{ \big\vert N_t^{- B_j - \frac{\beta}{2} t} \big\vert = l_j \big\} \ \Big\vert \ \mathcal{F}_s \Big) \nonumber \\
\to &\prod_{i = 1}^{n} \Big( \frac{\big( \mu(A_i) M_\infty \big)^{k_i}}{k_i !} 
\mathrm{e}^{- \mu(A_i) M_\infty} \Big) \prod_{j = 1}^{m} \Big( \frac{\big( \mu(B_j) M_\infty \big)^{l_j}}{l_j !} 
\mathrm{e}^{- \mu(B_j) M_\infty} \Big) \qquad P^{x_0} \text{-a.s.}
\end{align}
as $t \to \infty$.
\begin{proof}
For every particle $u \in N_s$ and a set $D \subseteq \mathbb{R}$ we define 
\[
N_{t-s}^D(u) := \{v \in N_t : X^v_t \in D \text{ and } u < v\},
\]
the set of descendants of $u$ at time $t$ whose spatial position at time $t$ belongs to the set $D$. 

Let us fix any number $K > \frac{\beta}{2}$ and define the following two events:
\begin{align*}
S_t^1 := &\bigcap_{u \in N_s} \Big\{ |X^u_s| < Ks \Big\} \bigcap \Big\{ |N_s| \geq k \Big\},\\
S_t^2 := &\bigcap_{u \in N_s} \Big\{ \big\vert N_{t-s}^{(A + \frac{\beta}{2} t)\cup(-B - \frac{\beta}{2} t)}(u) \big\vert \leq 1\Big\}.
\end{align*}
Then we already know from \eqref{old_supercrit} that $P^{x_0}\big(S_t^1\big) \to 1$ as $t \to \infty$ (eventually, all the particles are contained in $(-Ks, Ks)$ at time $s$ and the total number of particles at time $s$ increases to $\infty$). Thus 
\begin{equation}
\label{s1}
\mathbf{1}_{S_t^1} \to 1 \qquad P^{x_0}-\text{a.s.}
\end{equation}
as $t \to \infty$. Also from estimate \eqref{estimate4} we have that
\begin{align*}
P^{x_0} \Big( \big( S_t^2 \big)^c \Big\vert \mathcal{F}_s \Big) 
= &P^{x_0} \Big( \bigcup _{u \in N_s} \Big\{ \big\vert N^{(A + \frac{\beta}{2} t)\cup(-B - \frac{\beta}{2} t)}_{t-s}(u) \big\vert > 1 \Big\} \Big\vert \mathcal{F}_s\Big)\\
\leq &\sum_{u \in N_s} P^{X^u_s} \Big( \big\vert N^{(A + \frac{\beta}{2} t)\cup(-B - \frac{\beta}{2} t)}_{t-s} \big\vert > 1 \Big)\\
\leq &C\sum_{u \in N_s} \mathrm{e}^{- \beta |X^u_s| - \beta^2 s}\\
= &C \mathrm{e}^{- \frac{\beta^2}{2}s} M_s\\
\to &0 \qquad P^{x_0} \text{-a.s. as } t \to \infty 
\end{align*}
and hence
\begin{align}
\label{s2}
P^{x_0} \Big( S_t^2 \Big\vert \mathcal{F}_s \Big) \to 1 \qquad P^{x_0} \text{-a.s.}
\end{align}
From \eqref{s1} and \eqref{s2} we have that 
\begin{align*}
&P^{x_0} \Big( \bigcap_{i=1}^n \big\{ \big\vert N_t^{A_i + \frac{\beta}{2} t} \big\vert = k_i \big\} \text{ ,}
\bigcap_{j=1}^{m} \big\{ \big\vert N_t^{- B_j - \frac{\beta}{2} t} \big\vert = l_j \big\} \ \Big\vert \mathcal{F}_s \Big)\\
= &P^{x_0} \Big( \bigcap_{i=1}^n \big\{ \big\vert N_t^{A_i + \frac{\beta}{2} t} \big\vert = k_i \big\} \text{ ,}
\bigcap_{j=1}^{m} \big\{ \big\vert N_t^{- B_j - \frac{\beta}{2} t} \big\vert = l_j \big\} , \ S_t^1, \ S_t^2 \ \Big\vert \mathcal{F}_s \Big) + \epsilon_t
\end{align*}
for some $\epsilon_t$ such that $\epsilon_t \to 0$ $P^{x_0}$-a.s. We then note that on the event $S_t^2$ random variables $\big\vert N^{(A + \frac{\beta}{2} t)\cup(-B - \frac{\beta}{2} t)}_{t-s}(u) \big\vert $, $u \in N_s$ are Bernoulli random variables so that counting in how many ways $k$ particles from $N_s$ can be assigned to $(A + \frac{\beta}{2} t)\cup(-B - \frac{\beta}{2} t)$ (and the remaining particles from $N_s$ assigned to $\big((A + \frac{\beta}{2} t)\cup(-B - \frac{\beta}{2} t)\big)^c$) in such a way that of these $k$ particles $k_1$ are assigned to $A_1 + \frac{\beta}{2}t$, $k_2$ are assigned to $A_2 + \frac{\beta}{2}t$, $\cdots$, $k_n$ are assigned to $A_n + \frac{\beta}{2}t$, $l_1$ are assigned to $-B_1 - \frac{\beta}{2}t$, $\cdots$, $l_m$ are assigned to $-B_m - \frac{\beta}{2}t$ gives us
\begin{align*}
&P^{x_0} \Big( \bigcap_{i=1}^n \big\{ \big\vert N_t^{A_i + \frac{\beta}{2} t} \big\vert = k_i \big\} \text{ ,}
\bigcap_{j=1}^{m} \big\{ \big\vert N_t^{- B_j - \frac{\beta}{2} t} \big\vert = l_j \big\} , \ S_t^1, \ S_t^2 \ \Big\vert \mathcal{F}_s \Big)\\
= &\frac{1}{k_1! \cdots k_n! l_1! \cdots l_m!} P^{x_0} \Big( \bigcup_{(u_1, \cdots, u_k) \subseteq N_s} \Big\{ \big\vert N_{t-s}^{A_1 + \frac{\beta}{2} t}(u_1)\big\vert = 1, \cdots, \big\vert N_{t-s}^{A_1 + \frac{\beta}{2} t}(u_{k_1})
\big\vert = 1,\\ 
&\quad\qquad\qquad\qquad\qquad\qquad\qquad\qquad\qquad\qquad\qquad\qquad\qquad\cdots, \big\vert N_{t-s}^{- B_m - \frac{\beta}{2} t}(u_k)\big\vert = 1,\\
&\qquad\qquad\qquad\qquad\qquad\qquad\qquad  \big\vert 
N_{t-s}^{(A+\frac{\beta}{2} t) \cup (-B -\frac{\beta}{2} t)}(u)\big\vert = 0, 
u \neq u_1, \cdots, u_k\Big\} , \ S_t^1, \ S_t^2 \ \Big\vert \mathcal{F}_s \Big),
\end{align*}
where $\bigcup_{(u_1, \cdots, u_k) \subseteq N_s}$ is the union  over all $k$-permutations of $N_s$. Equivalently, $(u_1, \cdots, u_k) \subseteq N_s$ may be written as $u_1 \in N_s, \ u_2 \in N_s: u_2 \neq u_1, \ u_3 \in N_s: u_3 \neq u_1, u_2, \ \cdots, \ u_k \in N_s: u_k \neq u_1, \cdots, u_{k-1}$.

Then noting that $\bigcup_{(u_1, \cdots, u_k) \subseteq N_s} \{ \cdot \}$ is a union of mutually-disjoint events and that $|N_{t-s}^{(\cdot)}(u)|$, $u \in N_s$ are independent conditional on $\mathcal{F}_s$ we have that
\begin{align*}
&P^{x_0} \Big( \bigcup_{(u_1, \cdots, u_k) \subseteq N_s} \Big\{ \big\vert N_{t-s}^{A_1 + \frac{\beta}{2} t}(u_1)\big\vert = 1, \cdots, \big\vert N_{t-s}^{- B_m - \frac{\beta}{2} t}(u_k)
\big\vert = 1,\\ 
& \qquad\qquad\qquad\qquad\qquad  \big\vert 
N_{t-s}^{(A+\frac{\beta}{2} t) \cup (-B -\frac{\beta}{2} t)}(u)\big\vert = 0, 
u \neq u_1, \cdots, u_k\Big\} , \ S_t^1, \ S_t^2 \ \Big\vert \mathcal{F}_s \Big)\\
= &P^{x_0} \Big( \bigcup_{(u_1, \cdots, u_k) \subseteq N_s} \Big\{ \big\vert N_{t-s}^{A_1 + \frac{\beta}{2} t}(u_1)\big\vert = 1, \cdots, \big\vert N_{t-s}^{- B_m - \frac{\beta}{2} t}(u_k)
\big\vert = 1,\\ 
& \qquad\qquad\qquad\qquad\qquad  \big\vert 
N_{t-s}^{(A+\frac{\beta}{2} t) \cup (-B -\frac{\beta}{2} t)}(u)\big\vert = 0, 
u \neq u_1, \cdots, u_k\Big\} , \ S_t^1 \ \Big\vert \mathcal{F}_s \Big) + \epsilon_t'\\
= &\mathbf{1}_{S_t^1} \sum_{(u_1, \cdots, u_k) \subseteq N_s} \Big[ P^{X^{u_1}_s} \big( \big\vert N_{t-s}^{A_1 + \frac{\beta}{2} t}\big\vert = 1 \big) \times \cdots \times P^{X^{u_k}_s} \big( \big\vert N_{t-s}^{- B_m - \frac{\beta}{2} t}\big\vert = 1 \big)\\
& \qquad\qquad\qquad\qquad \times \prod_{u \neq u_1, \cdots, u_k } P^{X^{u}_s} \big( \big\vert N_{t-s}^{(A + \frac{\beta}{2} t)\cup(-B - \frac{\beta}{2} t)}\big\vert = 0 \big)\Big] + \epsilon_t',
\end{align*}
where $\epsilon_t' \to 0$ $P^{x_0}$-a.s.

We have thus shown so far that 
\begin{align*}
&P^{x_0} \Big( \bigcap_{i=1}^n \big\{ \big\vert N_t^{A_i + \frac{\beta}{2} t} \big\vert = k_i \big\} \text{ ,}
\bigcap_{j=1}^{m} \big\{ \big\vert N_t^{- B_j - \frac{\beta}{2} t} \big\vert = l_j \big\} \ \Big\vert \mathcal{F}_s \Big)\\
=&\mathbf{1}_{S_t^1} \frac{1}{k_1! \cdots k_n! l_1! \cdots l_m!} \prod_{u \in N_s } P^{X^{u}_s} \big( \big\vert N_{t-s}^{(A + \frac{\beta}{2} t)\cup(-B - \frac{\beta}{2} t)}\big\vert = 0 \big)\\
\times &\sum_{(u_1, \cdots, u_k) \subseteq N_s} \Big[ \frac{P^{X^{u_1}_s} \big( \big\vert N_{t-s}^{A_1 + \frac{\beta}{2} t}\big\vert = 1 \big)}{P^{X^{u_1}_s} \big( \big\vert N_{t-s}^{(A + \frac{\beta}{2} t)\cup(-B - \frac{\beta}{2} t)}\big\vert = 0 \big)} \ 
\cdots \ \frac{P^{X^{u_k}_s} \big( \big\vert N_{t-s}^{- B_m - \frac{\beta}{2} t}\big\vert = 1 \big)}{P^{X^{u_k}_s} \big( \big\vert N_{t-s}^{(A + \frac{\beta}{2} t)\cup(-B - \frac{\beta}{2} t)}\big\vert = 0 \big)} \Big] + \epsilon_t''
\end{align*}
for some $\epsilon_t''$ such that $\epsilon_t'' \to 0$ $P^{x_0}$-a.s. To establish \eqref{eq_main_proof} we shall show that $P^{x_0}$-almost surely
\begin{equation}
\label{limit1}
\lim_{t \to \infty} \mathbf{1}_{S_t^1} \prod_{u \in N_s } P^{X^{u}_s} \big( \big\vert N_{t-s}^{(A + \frac{\beta}{2} t)\cup(-B - \frac{\beta}{2} t)}\big\vert = 0 \big) = \mathrm{e}^{-(\mu(A_1)+ \cdots + \mu(B_m))M_\infty} 
= \mathrm{e}^{-( \mu(A) + \mu(B)) M_\infty}
\end{equation}
and
\begin{align}
\label{limit2}
&\lim_{t \to \infty} \mathbf{1}_{S_t^1} \sum_{(u_1, \cdots, u_k) \subseteq N_s} \frac{P^{X^{u_1}_s} \big( \big\vert N_{t-s}^{A_1 + \frac{\beta}{2} t}\big\vert = 1 \big)}{P^{X^{u_1}_s} \big( \big\vert N_{t-s}^{(A + \frac{\beta}{2} t)\cup(-B - \frac{\beta}{2} t)}\big\vert = 0 \big)} \cdots \frac{P^{X^{u_k}_s} \big( \big\vert N_{t-s}^{- B_m - \frac{\beta}{2} t}\big\vert = 1 \big)}{P^{X^{u_k}_s} \big( \big\vert N_{t-s}^{(A + \frac{\beta}{2} t)\cup(-B - \frac{\beta}{2} t)}\big\vert = 0 \big)}\nonumber\\
= &\prod_{i = 1}^{n} \big( \mu(A_i) M_\infty \big)^{k_i} \prod_{j = 1}^{m} \big( \mu(B_j) M_\infty \big)^{l_j}.
\end{align}
Then since $\epsilon_t'' \to 0$ we will get the sought result.

\underline{Proof of \eqref{limit1}}:

From \eqref{estimate2a} and the trivial fact that $1-x \leq \exp \{-x\}$ for all $x \in \mathbb{R}$ it follows that on the event $S_t^1$
\begin{align}
\label{limit1_upper}
&\prod_{u \in N_s } P^{X^{u}_s} \big( \big\vert N_{t-s}^{(A + \frac{\beta}{2} t)\cup(-B - \frac{\beta}{2} t)}\big\vert = 0 \big)\nonumber\\
\leq &\prod_{u \in N_s} \Big( 1 - \big( \mu(A) + \mu(B) \big) \mathrm{e}^{- \beta |X^u_s| - \frac{\beta^2}{2} s} \theta_6(t) \Big)\nonumber\\
\leq &\exp \Big\{-\big( \mu(A) + \mu(B)\big) \theta_6(t) \sum_{u \in N_s}\mathrm{e}^{- \beta |X_u^s| - \frac{\beta^2}{2}s}\Big\}\nonumber\\
\to &\mathrm{e}^{-(\mu(A) + \mu(B))M_\infty} \qquad P^{x_0} \text{-a.s.}
\end{align}
as $t$ (and hence $s(t)$) $\to \infty$.

For the lower bound we use the fact that for all $x^\ast \in (0, 1)$ and $x \in [0, x^\ast)$
\[
\log(1-x) \geq \frac{\log(1-x^\ast)}{x^\ast}x
\]
and hence 
\[
1 - x \geq \exp \Big\{ \frac{\log(1-x^\ast)}{x^\ast} x\Big\}
\]
and that $\frac{\log(1 - x^\ast)}{x^\ast} \to -1$ as $x^\ast \searrow 0$. Taking $x^\ast = \big( \mu(A) + \mu(B)\big) \mathrm{e}^{-\frac{\beta^2}{2}s} \theta_5(t)$ and $x = \big( \mu(A) + \mu(B)\big) \mathrm{e}^{-\beta |X^u_s| -\frac{\beta^2}{2}s} \theta_5(t)$ we get from \eqref{estimate2} that on the event $S_t^1$
\begin{align}
\label{limit1_lower}
&\prod_{u \in N_s } P^{X^{u}_s} \big( \big\vert N_{t-s}^{(A + \frac{\beta}{2} t)\cup(-B - \frac{\beta}{2} t)}\big\vert = 0 \big)\nonumber\\
\geq &\prod_{u \in N_s} \Big( 1 - \big( \mu(A) + \mu(B) \big) \mathrm{e}^{- \beta |X^u_s| - \frac{\beta^2}{2} s} \theta_5(t) \Big)\nonumber\\
\geq &\exp \Big\{ \frac{\log \big(1 - ( \mu(A) + \mu(B)) \mathrm{e}^{-\frac{\beta^2}{2}s} \theta_5(t)\big)}{( \mu(A) + \mu(B)) \mathrm{e}^{-\frac{\beta^2}{2}s} \theta_5(t)}
\big( \mu(A) + \mu(B)\big) \theta_5(t) \sum_{u \in N_s}\mathrm{e}^{- \beta |X_u^s| - \frac{\beta^2}{2}s}\Big\}\nonumber\\
\to &\mathrm{e}^{-(\mu(A) + \mu(B))M_\infty} \qquad P^{x_0} \text{-a.s.}
\end{align}
as $t \to \infty$. Equations \eqref{limit1_upper} and \eqref{limit1_lower} together yield \eqref{limit1}.

\underline{Proof of \eqref{limit2}}:

From \eqref{estimate2} and \eqref{estimate3a} it follows that on the event $S_t^1$ 
\begin{align}
\label{limit2_upper}
&\sum_{(u_1, \cdots, u_k) \subseteq N_s} \frac{P^{X^{u_1}_s} \big( \big\vert N_{t-s}^{A_1 + \frac{\beta}{2} t}\big\vert = 1 \big)}{P^{X^{u_1}_s} \big( \big\vert N_{t-s}^{(A + \frac{\beta}{2} t)\cup(-B - \frac{\beta}{2} t)}\big\vert = 0 \big)} \cdots \frac{P^{X^{u_k}_s} \big( \big\vert N_{t-s}^{- B_m - \frac{\beta}{2} t}\big\vert = 1 \big)}{P^{X^{u_k}_s} \big( \big\vert N_{t-s}^{(A + \frac{\beta}{2} t)\cup(-B - \frac{\beta}{2} t)}\big\vert = 0 \big)}\nonumber\\
\leq &\sum_{(u_1, \cdots, u_k) \subseteq N_s} 
\frac{\mu(A_1) \mathrm{e}^{- \beta |X^{u_1}_s| - \frac{\beta^2}{2}s} \theta_8(t)}{1 - \big( \mu(A) + \mu(B)\big) \mathrm{e}^{- \frac{\beta^2}{2}s} \theta_5(t)} \cdots
\frac{\mu(B_m) \mathrm{e}^{- \beta |X^{u_k}_s| - \frac{\beta^2}{2}s} \theta_8(t)}{1 - \big( \mu(A) + \mu(B)\big) \mathrm{e}^{- \frac{\beta^2}{2}s} \theta_5(t)}\nonumber\\
\leq &\Big[ \frac{\theta_8(t)}{1 - \big( \mu(A) + \mu(B)\big) \mathrm{e}^{- \frac{\beta^2}{2}s} \theta_5(t)} \Big]^k 
\Big[ \mu(A_1) \sum_{u_1 \in N_s} \mathrm{e}^{- \beta |X^{u_1}_s| - \frac{\beta^2}{2}s} \Big] 
\times \cdots\nonumber\\
&\qquad\qquad\qquad\qquad\qquad\qquad\quad \times \Big[ \mu(B_m) \sum_{u_k \in N_s} \mathrm{e}^{- \beta |X^{u_k}_s| - \frac{\beta^2}{2}s} \Big]\nonumber\\
\to &\prod_{i = 1}^{n} \big( \mu(A_i) M_\infty \big)^{k_i} \prod_{j = 1}^{m} \big( \mu(B_j) M_\infty \big)^{l_j} \qquad P^{x_0} \text{-a.s.}
\end{align}
as $t \to \infty$. For the lower bound we notice that 
\begin{align*}
M_s^k = &\sum_{u_1, \cdots, u_k \in N_s} \Big[ \mathrm{e}^{- \beta |X^{u_1}_s| - \frac{\beta^2}{2}s} \cdots \mathrm{e}^{- \beta |X^{u_k}_s| - \frac{\beta^2}{2}s} \Big]\\
\leq &\sum_{(u_1, \cdots, u_k) \subseteq N_s} \Big[ \mathrm{e}^{- \beta |X^{u_1}_s| - \frac{\beta^2}{2}s} \cdots \mathrm{e}^{- \beta |X^{u_k}_s| - \frac{\beta^2}{2}s} \Big]\\ 
+ &\sum_{1 \leq i < j \leq k} \sum_{\substack{u_1, \cdots, u_k \in N_s:\\ u_i = u_j}} \Big[  \mathrm{e}^{- \beta |X^{u_1}_s| - \frac{\beta^2}{2}s} \cdots \mathrm{e}^{- \beta |X^{u_k}_s| - \frac{\beta^2}{2}s} \Big]\\
\leq &\sum_{(u_1, \cdots, u_k) \subseteq N_s} \Big[ \mathrm{e}^{- \beta |X^{u_1}_s| - \frac{\beta^2}{2}s} \cdots \mathrm{e}^{- \beta |X^{u_k}_s| - \frac{\beta^2}{2}s} \Big] + \binom{k}{2} \mathrm{e}^{- \frac{\beta^2}{2}s} M_s^{k-1}
\end{align*}
Then from \eqref{estimate3} we have that on the event $S_t^1$
\begin{align}
\label{limit2_lower}
&\sum_{(u_1, \cdots, u_k) \subseteq N_s} \frac{P^{X^{u_1}_s} \big( \big\vert N_{t-s}^{A_1 + \frac{\beta}{2} t}\big\vert = 1 \big)}{P^{X^{u_1}_s} \big( \big\vert N_{t-s}^{(A + \frac{\beta}{2} t)\cup(-B - \frac{\beta}{2} t)}\big\vert = 0 \big)} \cdots \frac{P^{X^{u_k}_s} \big( \big\vert N_{t-s}^{- B_m - \frac{\beta}{2} t}\big\vert = 1 \big)}{P^{X^{u_k}_s} \big( \big\vert N_{t-s}^{(A + \frac{\beta}{2} t)\cup(-B - \frac{\beta}{2} t)}\big\vert = 0 \big)}\nonumber\\
\geq &\sum_{(u_1, \cdots, u_k) \subseteq N_s} \Big[
\mu(A_1) \mathrm{e}^{- \beta |X^{u_1}_s| - \frac{\beta^2}{2}s} \theta_7(t) \Big] \cdots \Big[
\mu(B_m) \mathrm{e}^{- \beta |X^{u_k}_s| - \frac{\beta^2}{2}s} \theta_7(t) \Big]\nonumber\\
= &\theta_7(t)^k \prod_{i = 1}^{n} \big( \mu(A_i) \big)^{k_i} \prod_{j = 1}^{m} \big( \mu(B_j) \big)^{l_j}
\sum_{(u_1, \cdots, u_k) \subseteq N_s} \Big[ \mathrm{e}^{- \beta |X^{u_1}_s| - \frac{\beta^2}{2}s} \cdots \mathrm{e}^{- \beta |X^{u_k}_s| - \frac{\beta^2}{2}s} \Big]\nonumber\\
\geq &\theta_7(t)^k \prod_{i = 1}^{n} \big( \mu(A_i) \big)^{k_i} \prod_{j = 1}^{m} \big( \mu(B_j) \big)^{l_j}
\Big( M_s^k - \binom{k}{2} \mathrm{e}^{- \frac{\beta^2}{2}s} M_s^{k-1}\Big)\nonumber\\
\to &\prod_{i = 1}^{n} \big( \mu(A_i) M_\infty \big)^{k_i} \prod_{j = 1}^{m} \big( \mu(B_j) M_\infty \big)^{l_j} 
\qquad P^{x_0} \text{-a.s.}
\end{align}
as $t \to \infty$. Upper bound \eqref{limit2_upper} and lower bound \eqref{limit2_lower} together establish  \eqref{limit2}, which completes the proof of \eqref{eq_main_proof}.
\end{proof}
\subsection{Proof of Theorem \ref{main3}}
Take any $x_0 \in \mathbb{R}$, $\lambda \in (0, \frac{\beta}{2})$ and a Borel set $A$ such that $\inf A > - \infty$. Our aim is to prove  that 
\begin{equation}
\label{eq_main3_proof}
\mathrm{e}^{- \Delta_\lambda t} |N_t^{A+\lambda t}| \to \mu(A) M_\infty \qquad 
\text{ in }P^{x_0} \text{-probability}.
\end{equation}
\begin{proof}
Let us first recall that by the Markov property for any $s \in [0,t]$
\[
E^{x_0} \Big( |N_t^{A+ \lambda t}| \big\vert \mathcal{F}_s\Big) = \sum_{u \in N_s} E^{X^u_s} |N_{t-s}^{A + \lambda t}|.
\]
Now let us take $s:[0, \infty) \to [0, \infty)$ to be such that $s(t) \to \infty$ but $s(t) = o(t)$ as $t \to \infty$ and $t - s(t) \geq 0$ for all $t \geq 0$. Then from \eqref{estimate1} we get
\[
\mathrm{e}^{- \Delta_\lambda t} E^{x_0} \Big( |N_t^{A+ \lambda t}| \big\vert \mathcal{F}_s\Big) 
\leq \mu(A) \sum_{u \in N_s} \mathrm{e}^{- \beta |X^u_s| - \frac{\beta^2}{3}s} \theta_2(t)
= \mu(A) \theta_2(t) M_s.
\]
Similarly
\[
\mathrm{e}^{- \Delta_\lambda t} E^{x_0} \Big( |N_t^{A+ \lambda t}| \big\vert \mathcal{F}_s\Big) 
\geq \mu(A) \theta_1(t) M_s
\]
and hence 
\begin{equation}
\label{almost_sure}
\mathrm{e}^{- \Delta_\lambda t} E^{x_0} \Big( |N_t^{A+ \lambda t}| \big\vert \mathcal{F}_s\Big) \to \mu(A) M_\infty 
\qquad P^{x_0} \text{-a.s.}
\end{equation}
as $t \to \infty$. On the other hand, for any choice of $\epsilon > 0$ we have
\begin{align*}
&P^{x_0} \Big( \mathrm{e}^{-\Delta_\lambda t} \Big\vert |N_t^{A + \lambda t}| - E^{x_0} \big( |N_t^{A + \lambda t}| \big\vert \mathcal{F}_s\big)\Big\vert > \epsilon \Big)\\
\leq &\frac{1}{\epsilon^2} \mathrm{e}^{-2 \Delta_\lambda t} E^{x_0} \Big[ |N_t^{A + \lambda t}| - E^{x_0} \big( |N_t^{A + \lambda t}| \big\vert \mathcal{F}_s\big) \Big]^2\\
= &\frac{1}{\epsilon^2} \mathrm{e}^{-2 \Delta_\lambda t} E^{x_0} \Big[ E^{x_0} \Big( \Big(
|N_t^{A + \lambda t}| - E^{x_0} \big( |N_t^{A + \lambda t}| \big\vert \mathcal{F}_s\big)
\Big)^2 \Big\vert \mathcal{F}_s \Big) \Big]\\
= &\frac{1}{\epsilon^2} \mathrm{e}^{-2 \Delta_\lambda t} E^{x_0} \Big[ E^{x_0} \Big( |N_t^{A + \lambda t}|^2 \big\vert \mathcal{F}_s \Big) - \Big( E^{x_0} \big( |N_t^{A + \lambda t}| \big\vert 
\mathcal{F}_s \big) \Big)^2 \Big].
\end{align*}
Then by the Markov property again 
\begin{align*}
&E^{x_0} \Big( |N_t^{A + \lambda t}|^2 \big\vert \mathcal{F}_s \Big) - \Big( E^{x_0} \big( |N_t^{A + \lambda t}| 
\big\vert \mathcal{F}_s \big) \Big)^2\\
= &\Big( \sum_{u \in N_s} E^{X^u_s} \big\vert N_{t-s}^{A + \lambda t} \big\vert^2 + \sum_{\substack{u, v \in N_s\\u \neq v}} \big( E^{X^u_s} \big\vert N_{t-s}^{A + \lambda t} \big\vert E^{X^v_s} \big\vert N_{t-s}^{A + \lambda t} \big\vert \big)\Big)\\
- &\Big( \sum_{u \in N_s} \big( E^{X^u_s} \big\vert N_{t-s}^{A + \lambda t} \big\vert \big)^2 + \sum_{\substack{u, v \in N_s\\u \neq v}} \big( E^{X^u_s} \big\vert N_{t-s}^{A + \lambda t} \big\vert E^{X^v_s} \big\vert N_{t-s}^{A + \lambda t} \big\vert \big)\Big)\\
= &\sum_{u \in N_s} \Big[ E^{X^u_s} \big\vert N_{t-s}^{A + \lambda t} \big\vert^2 - \big( E^{X^u_s} \big\vert N_{t-s}^{A + \lambda t} \big\vert \big)^2 \Big]\\
\leq &\sum_{u \in N_s} E^{X^u_s} \big\vert N_{t-s}^{A + \lambda t} \big\vert^2.
\end{align*}
Thus first applying \eqref{square_estimate_eq} and after that \eqref{estimate1} we get
\begin{align}
\label{in_probability}
&P^{x_0} \Big( \mathrm{e}^{-\Delta_\lambda t} \Big\vert |N_t^{A + \lambda t}| - E^{x_0} \big( |N_t^{A + \lambda t}| \big\vert \mathcal{F}_s\big)\Big\vert > \epsilon \Big)\nonumber\\
\leq &\frac{1}{\epsilon^2} \mathrm{e}^{-2 \Delta_\lambda t} E^{x_0} \Big( \sum_{u \in N_s} \Big[ E^{X^u_s} \big\vert N_{t-s}^{A + \lambda t} \big\vert + C \mathrm{e}^{- \beta |X^u_s| - \beta^2 s + 2 \Delta_\lambda t} \Big] \Big)\nonumber\\
\leq &\frac{1}{\epsilon^2} \mathrm{e}^{-2 \Delta_\lambda t} E^{x_0} \Big( \sum_{u \in N_s} \Big[ \mu(A) \mathrm{e}^{- \beta |X^u_s| - \frac{\beta^2}{2} s + \Delta_\lambda t}\theta_2(t) + C \mathrm{e}^{- \beta |X^u_s| - \beta^2 s + 2 \Delta_\lambda t} \Big] \Big)\nonumber\\
= &\frac{1}{\epsilon^2} \mathrm{e}^{-2 \Delta_\lambda t} E^{x_0} \Big( \mu(A) \theta_2(t) \mathrm{e}^{\Delta_\lambda t}M_s + C \mathrm{e}^{- \frac{\beta^2}{2}s + 2 \Delta_\lambda t}M_s \Big)\nonumber\\
= &\frac{1}{\epsilon^2} \mu(A) \theta_2(t) \mathrm{e}^{- \Delta_\lambda t} + \frac{1}{\epsilon^2} \mathrm{e}^{- \frac{\beta^2}{2}s} \to 0
\end{align}
as $t \to \infty$. Thus 
\[
\mathrm{e}^{- \Delta_\lambda t} \Big( \big\vert N_t^{A+\lambda t} \big\vert - E^{x_0} \big( |N_t^{A+\lambda t}| 
\big\vert \mathcal{F}_s \big) \Big) \to 0 \qquad \text{ in }P^{x_0} \text{-probability},
\]
which together with \eqref{almost_sure} proves Theorem \ref{main3}
\end{proof}
\begin{Remark}
\label{almost_sure}
Note that from the above estimate, convergence in \eqref{in_probability} can be seen to hold almost surely along sequences $(t_n)_{n \geq 1}$ such that $\sum_{n \geq 1} \mathrm{e}^{- \frac{\beta^2}{2}s(t_n) < \infty}$ and hence convergence in \eqref{eq_main3_proof} holds almost surely along sequences $(t_n)_{n \geq 1}$ such that 
$\frac{t_n}{(\log n)^\alpha} \to \infty$ for $\alpha > \frac{2}{\beta^2}$.
\end{Remark}

\end{document}